\documentclass[12pt]{amsart}

\usepackage[centering, margin={1.0in, 1.0in}, includeheadfoot]{geometry}

\usepackage{mathrsfs}
\usepackage{amsfonts}
\usepackage{tikz}
\usepackage{amssymb,bm}
\usepackage{amsmath}
\usepackage{amsthm}
\usepackage{mathtools}
\usepackage{tabularx}
\usepackage{hyperref}

\newtheorem{thm}{Theorem}[section]

\newtheorem{lem}[thm]{Lemma}
\newtheorem{cor}[thm]{Corollary}

\newtheorem{problem}[thm]{Problem}

\newcommand{\F}{\mathbb{F}_{q}}

\newcommand{\Zqn}{\mathbb{Z}_{q^n-1}}

\newcommand{\Fn}{\mathbb{F}_{q^n}}
\newcommand{\Tn}{\operatorname{Tr}_{\mathbb{F}_{q^n}/\F}}

\newcommand{\supp}{\operatorname{supp}}

\begin{document}
 \title{Irreducible polynomials with prescribed sums of coefficients}

 \author{Aleksandr Tuxanidy and Qiang Wang}

\address{School of Mathematics and Statistics, Carleton
University,
 1125 Colonel By Drive, Ottawa, Ontario, K1S 5B6,
Canada.} 

\email{AleksandrTuxanidyTor@cmail.carleton.ca, wang@math.carleton.ca}
 

\thanks{The research of Qiang Wang is partially supported by NSERC of Canada.}
\date{\today}

 \font\Bbb msbm10 at 12pt

\begin{abstract}

Let $q$ be a power of a prime, let $\mathbb{F}_q$ be the finite field with $q$ elements and let $n \geq 2$. For a polynomial $h(x) \in \F[x]$ of degree $n \in \mathbb{N}$ 
and a subset $W \subseteq [0,n] := \{0, 1, \ldots, n\}$, 
we define the sum-of-digits function
$$S_W(h) = \sum_{w \in W}[x^{w}] h(x)$$ 
to be the sum of all the coefficients of $x^w$ in $h(x)$ with $w \in W$. 
In the case when $q = 2$, we prove, except for a few genuine exceptions, that for any $c \in \mathbb{F}_2$ and any $W \subseteq [0,n]$
there exists an irreducible polynomial $P(x)$ of degree $n$ over $\mathbb{F}_2$ 
such that $S_{W}(P) = c$.
In particular, restricting ourselves to the case when $\# W = 1$,
we obtain a new proof of the Hansen-Mullen irreducibility conjecture (now a theorem) in the case when $q = 2$. 
In the case of $q> 2$, we prove that, for any $c \in \mathbb{F}_q$, any $n\geq 2$ and any $W \subseteq [0,n]$, there exists an irreducible polynomial $P(x)$ of degree $n$ such that $S_{W}(P) \neq c$. 

\end{abstract}

 \maketitle

 \section{Introduction}

 Let $q$ be a power of a prime $p$, let $\F$ be the finite field with $q$ elements, and let $n \geq 2$.
 In 1992, Hansen-Mullen \cite{hansen-mullen} conjectured (Conjecture B) that, except for a few genuine exceptions, 
 there exist monic irreducible (and more strongly primitive; see Conjecture A) polynomials of degree $n$ over $\F$ with 
 any {\em one} of its coefficients prescribed to any value. Conjecture B was proven by Wan \cite{wan} in 1997 for all but finitely many cases, 
 with the remaining cases being computationally verified soon after in \cite{ham-mullen}.
In 2006, Cohen \cite{cohen 2006}, particularly building on some of the work of Fan-Han \cite{fan-han} on $p$-adic series, proved 
there exists a monic primitive polynomial of degree $n \geq 9$ with any one of its coefficients prescribed. 
The remaining cases of Conjecture A were settled by Cohen-Pre\v{s}ern in \cite{cohen-presern 2006, cohen-presern 2008}.  Recently the authors \cite{Tuxanidy-Wang-H-M conjecture} reproved Conjecture B in an elementary way, through  studying an interesting connection between irreducible polynomials of degree $n$ over $\F$ 
and the least period of the discrete Fourier transform (DFT) of cyclic functions 
with values in a finite field.

Natural generalizations of the Hansen--Mullen conjectures to {\em several} prescribed coefficients is currently an active area of research.
For irreducible polynomials of degree $n$ over $\F$, 
Garefalakis \cite{garefalakis} has shown that one can prescribe 
roughly $n/3$ consecutive zero coefficients. Panario--Tzanakis \cite{George paper} (see also
\cite{George thesis}) have in particular proved that if $n \geq 22$ and $q \geq 107$, then one can arbitrarily prescribe
both the first coefficient and another coefficient. 
In 2013 Pollack \cite{pollack} has shown that for large enough $n$, one can prescribe roughly $\sqrt{n}$ coefficients to any value.  Recently, Ha \cite{Ha}
showed that there is a monic irreducible polynomial of degree $n$ with $r$ coefficients prescribed in any location when $r\leq [(1/4 - \epsilon)n]$ for any $\epsilon >0$ and $q$ is large; 
and when $r \leq \delta n$ for some $\delta>0$ and for any $q$.

In the special case of monic primitive polynomials, for sufficiently large $q$ (depending on $n$) 
it is known that up to the first $\lfloor n/2 \rfloor$ coefficients can be prescribed. See the work of Ren \cite{ren} and Han \cite{han} for this. 
Specifically in the case when $q = 2$, Shparlinski \cite{Shparlinski} showed that, for sufficiently large $n$ (in an unspecified manner) 
there exists a primitive polynomial of degree $n$ over $\mathbb{F}_2$ with Hamming weight 
(i.e., number of non-zero coefficients) $n/4 + o(n)$.
Cohen \cite{Cohen 2004} later showed in particular, also in the case of $q = 2$, 
that we can prescribe either the first or last $m \leq n/4$ coefficients of primitive polynomials of degree $n$ (for any $n$) over $\mathbb{F}_2$ to any value. 

There are some differences of approach in tackling existence questions of either general irreducible or primitive polynomials with prescibed coefficients. 
For instance, when working on irreducibles, and following in the footsteps of Wan \cite{wan}, 
it has been common practice to exploit the function field analogue of Dirichlet's theorem for primes in arithmetic progressions; 
all this is done via Dirichlet characters on $\F[x]$, $L$-series, zeta functions, etc.
See for instance \cite{George thesis}.
On the other hand, in the case of primitives, the problem is usually approached via $p$-adic rings or fields 
(to account for the inconvenience that Newton's identities break down in fields of positive characteristic) 
together with Cohen's sieving lemma, Vinogradov's characteristic function, etc. (see for example \cite{fan-han, cohen 2006}).
However there is one common feature these two methods share, namely, when bounding
the ``error'' terms comprised of character sums, the function field analogue of Riemann's hypothesis (Weil's bound) is used. 
Nevertheless as a consequence of its $O(q^{n/2})$ nature it transpires 
a difficulty in extending the $n/2$ threshold for the number of coefficients one can prescribe in irreducible or particularly primitive polynomials of degree $n$. 

In this work we consider the following related problem. 
First for a polynomial $h(x) \in \F[x]$ of degree $n \in \mathbb{N}$ and $W \subseteq [0,n] = \{0, 1, \ldots, n\}$, we define the sum-of-digits function $S_W(h)$ by 
$$
S_W(h) := \sum_{w \in W}[x^{w}]h(x),
$$
the sum of all the coefficients of $x^w$ in $h(x)$ such that $w \in W$. 

\begin{problem}\label{problem}
Let $n \geq2$. For what elements $c \in \F$ and sets $W \subseteq [0,n]$ can we find a monic irreducible polynomial $P(x)$ of degree $n$ over $\F$ 
such that $S_W(P) = c$?
\end{problem}

Obviously if one can prove existence of monic irreducible polynomials with  prescribed coefficients for any $W$, then Problem \ref{problem} follows automatically. 
For example, the work in \cite{Ha} implies that whenever $W$ is roughly of cardinality $< n/4$, that we can prescribe $S_W(h)$ to any value. 
For $q$ large enough (depending on $n$), the result follows whenever $\#W \leq \lfloor n/2 \rfloor$ (see \cite{han, ren}). In the case when $q=2$ and $\#W \leq n/4$ with $W \subset [0, n/2)$ 
or $W \subset [n/2, n)$, Problem \ref{problem} follows automatically from \cite{Cohen 2004}. 
In fact, Problem \ref{problem} is much less ambitious than the  one of prescribing several coefficients.
Nevertheless when $W$ is sufficiently large 
(say $\#W > n/2$ roughly) Problem \ref{problem} is, to the knowledge of the authors, unsolved. 
We expect in most cases, except perhaps for some genuine exceptions, 
that $S_W(P)$, where $P(x)$ runs over monic irreducibles of degree $n$ over $\F$, 
can be prescribed to any value $c \in \F$ for any $W \subseteq [0,n]$.

The sum-of-digits function $S_w(h)$ for $\F[x]$ bears some resemblance to the function $s_{b, \mathcal{I}}(n)$ for $\mathbb{N}$ of the sum of $b^i$-digits, $i \in \mathcal{I} \subseteq \mathbb{N}_0$,
in the base $ b \geq 2$ 
expansion of $n \in \mathbb{N}$. We can write $n$ uniquely as 
$$n = \sum_{i=0}^\infty a_i b^i$$
with each $0 \leq a_i \leq b-1$.
Then
$$s_{b, \mathcal{I}}(n) := \sum_{i \in \mathcal{I}} a_i.$$
In the special case when $\mathcal{I} = \mathbb{N}_0$, 
the question regarding the distribution of the values of $s_{b, \mathbb{N}_0}(\ell)$ at the primes $\ell \in \mathbb{N}$, with $\ell$ at most a given $X \in \mathbb{R}$, 
has attracted substantial research. For instance Mauduit-Rivat \cite{Mauduit-Rivat} 
recently proved that the values of $s_{b, \mathbb{N}_0}(\ell)$ are asymptotically evenly distributed at the prime numbers $\ell$. 
For example, on average, there are as many prime numbers for which the sum of its decimal digits is even as prime numbers for which the sum is odd.
See also \cite{Drmota-Rivat-Stoll, Morgenbesser} for the analogy in the ring of Gaussian integers $\mathbb{Z}[\sqrt{-1}]$. 
Note it has been raised as an open problem in \cite{Koc} (see Open Problem 18, \S 4.4, p.123) to give an $\F[x]$-analogue 
of the sum-of-digits $s_{b, \mathbb{N}_0}(\ell)$ problem in \cite{Mauduit-Rivat}.

In this work we completely settle Problem \ref{problem} in the case when $q = 2$. 

 \begin{thm}\label{thm: main q = 2}
 Let $n \geq 2$, let $c \in \mathbb{F}_2$ and let $W \subseteq [0,n]$. 
 Then there exists an irreducible polynomial $P(x)$ of degree $n$ over $\mathbb{F}_2$ such that $S_{W}(P) = c$ if and only if 
 \begin{align}\label{eqn: exc1}
 (c, W) \neq \  &(0, \{0\}), \ (0, \{n\}), \ (0, [0,n]), \ (0, [1,n-1]), \\
 & (1, \{0,n\}), \ (1, [0, n-1]), \ (1, [1,n]). \nonumber
 \end{align}
 \end{thm}
 
 Note Theorem \ref{thm: main q = 2} states that, except for a few genuine exceptions, 
 there is no subset $W \subset [0,n]$ for which all binary irreducible polynomials of degree $n$ have the same parity in the number of non-zero coefficients of $x^w$ for $w \in W$.
 As an immediate consequence we obtain the following.
 
 \begin{cor}[{\bf Hansen-Mullen irreducibility conjecture for $q = 2$}]\label{cor: Hansen-Mullen}
 Let $n \geq 2$, let $0 \leq w < n$ and let $c \in \mathbb{F}_2$. 
 Then there exists an irreducible polynomial $P(x)$ of degree $n$ over $\mathbb{F}_2$ such that $[x^w]P(x) = c$ except when $(w, c) = (0, 0)$ and $(n,w,c) = (2,1,0)$.
 \end{cor}

 In the case when $q > 2$, we give the following weaker analogue of Theorem \ref{thm: main q = 2}. We expect that the techniques developed here could be valuable to tackle the cases when $q > 2$ as well, but we leave this for a future work. 
 
 \begin{thm}\label{thm: main q > 2}
 Let $q > 2$ be a power of a prime, let $c \in \F$, let $n \geq 2$ and let $W \subseteq [0,n]$. Then there exists an monic irreducible polynomial $P(x)$ of degree $n$ over $\F$ such that
 $S_W(P) \neq c$. 
\end{thm}

One of the main ingredients in the proof of Theorem \ref{thm: main q = 2}, \ref{thm: main q > 2}, 
is a seemingly new sufficient condition for a function on $\Fn$ to have an element of degree $n$ over $\F$ in its support, studied in \cite{Tuxanidy-Wang-H-M conjecture}; see also Lemma \ref{thm: connection}.  This unexpected connection to irreducible polynomials of degree $n$ over $\F$ 
is made via the least period of the discrete Fourier transform (DFT) of cyclic functions 
with values in a finite field. 
We exploit this relation by proving, in Lemma \ref{lem: period 1}, that the DFT of linear combinations of characteristic elementary symmetric (CES) functions 
(which produce the coefficients of characteristic polynomials) have the maximum possible least period (except for a few genuine exceptions).
This bears a sharp contrast to previous techniques in literature employed to tackle existence of irreducible polynomials with prescribed coefficients. 

The rest of this work goes as follows. In Section \ref{section: dft} we recall some preliminary 
facts regarding the DFT, convolution, least period of cyclic functions, 
and give a sufficient condition in Lemma \ref{thm: connection} for an element of $\Fn$ to have degree $n$ over $\F$. 
Note Lemma \ref{thm: connection} also gives a necessary condition for primitive elements of $\Fn$ to be contained in the support of functions on $\Fn$, although we do not make use of this fact here.
In Section \ref{section: delta functions} we place the CES functions in the context of their DFT, which we refer to as delta functions; see Section \ref{section: delta functions}.
We then apply Lemma \ref{thm: connection} to give, 
in Lemma \ref{lem: gamma and delta functions}, sufficient conditions for the existence of irreducible polynomials with prescribed sum of coefficients. 
Finally in Section \ref{section: period of delta functions} we prove, in Lemma \ref{lem: period 1}, 
that the delta functions do indeed attain the maximum possible least period and hence prove the main results:  Theorem \ref{thm: main q = 2} and Theorem \ref{thm: main q > 2}.

\section{Preliminaries}\label{section: dft}
 We recall some preliminary concepts regarding the discrete Fourier transform (DFT) on finite fields, convolution, and the least period of functions on cyclic groups. 
 
Let  $N \in \mathbb{N}$ such that $N \mid q-1$, and let $\zeta_N$ be a primitive $N$-th root of unity in $\F^*$ 
 (the condition on $N$ guarantees the existence of $\zeta_N$). We shall use the common notation $\mathbb{Z}_N := \mathbb{Z}/ N \mathbb{Z}$.
 Now the {\em discrete Fourier transform} (DFT) based on $\zeta_N$, on the $\F$-vector space of functions $f : \mathbb{Z}_N \to \F$, 
 is defined by
 $$
 \mathcal{F}_{\zeta_N}[f](i) = \sum_{j \in \mathbb{Z}_N} f(j) \zeta_N^{ij}, \hspace{1em} i \in \mathbb{Z}_N.
 $$
 Note $\mathcal{F}_{\zeta_N}$ is a bijective linear operator with 
inverse given by $\mathcal{F}^{-1}_{\zeta_N} = N^{-1} \mathcal{F}_{\zeta_N^{-1}}$. 

For $f,g : \mathbb{Z}_N \to \F$, the convolution of $f,g$ is the function $f \otimes g : \mathbb{Z}_N \to \F$ given by
$$
(f\otimes g)(i) = \sum_{\substack{j + k = i \\ j,k \in \mathbb{Z}_N} } f(j)g(k). 
$$
Inductively, $f_1 \otimes f_2 \otimes \cdots \otimes f_k = f_1 \otimes (f_2 \otimes \cdots \otimes f_k)$ and so
$$
(f_1 \otimes \cdots \otimes f_k)(i) = \sum_{\substack{j_1 + \cdots + j_k = i \\ j_1, \ldots, j_k \in \mathbb{Z}_N}} f_1(j_1) \cdots f_k(j_k).
$$
For $m \in \mathbb{N}$, we let $f^{\otimes m}$ denote the $m$-th convolution power of $f$, that is, the convolution of $f$ with itself, $m$ times. 
The DFT and convolution are related by the fact that 
$$
\prod_{i=1}^k\mathcal{F}_{\zeta_N}[f_i] = \mathcal{F}_{\zeta_N}\left[ \bigotimes_{i=1}^k f_i\right]. 
$$
Since $f, \mathcal{F}_{\zeta_N}[f]$, have values in $\F$ by definition, it follows from the relation above that $f^{\otimes q} = f$.
Convolution is associative, commutative and distributive with identity $\delta_0 : \mathbb{Z}_N \to \{0, 1\} \subseteq \mathbb{F}_p$, the Kronecker delta function 
defined by $\delta_0(i) = 1$ if $i = 0$ and $\delta_0(i) = 0$ otherwise.
We set $f^{\otimes 0 } = \delta_0$.

Next we recall the concepts of a period and least period of a function $f:\mathbb{Z}_N \to \F$. 
For $r \in \mathbb{N}$, we say that $f$ is {\em $r$-periodic} if $f(i ) = f(i + \overline{r})$ for all $i \in \mathbb{Z}_N$. 
Clearly $f$ is $r$-periodic if and only if it is $\gcd(r, N)$-periodic.
The smallest such positive integer $r$ is called the {\em least period} of $f$. Note the least period $r$ satisfies $r \mid N$.
If the least period of $f$ is $N$, we say that $f$ has {\em maximum least period}.

There are various operations on cyclic functions which preserve the least period. For instance the {\em $k$-shift} function $f_k(i) := f(i + k)$ of $f$ has the same least period as $f$. 
The {\em reversal} function $f^*(i) := f(-(1 + i))$ of $f$ also has the same least period. Let $\sigma$ be a permutation of $\F$. The {\em permuted} function $f^{\sigma}(i) := \sigma(f(i))$ keeps the 
least period of $f$ as well. In particular if $f(\mathbb{Z}_N) \subseteq \{0,1\}$ and $\sigma$ sends $0$ to $1$ and $1$ to $0$, we call $f^{\sigma}$ the {\em complement} of $f$.

Let $\Phi_n(x) \in \mathbb{Z}[x]$ be the $n$-th cyclotomic polynomial. For a function $F$ on a set $A$, let $\supp(F) := \{a \in A \ : \ F(a) \neq 0\}$ be the support of $F$. The following result from \cite{Tuxanidy-Wang-H-M conjecture} seems to be quite useful. 

\begin{lem}\label{thm: connection}
 Let $q$ be a power of a prime, let $n \geq 2$, let $\zeta$ be a primitive element of $\Fn$, let $F : \Fn \to \Fn$, 
 let $f : \mathbb{Z}_{q^n-1} \to \Fn$ be defined by $f(k) = F(\zeta^k)$, and let $r$ be the least period of $\mathcal{F}_\zeta[f]$ 
 (which is the same as the least period of $\mathcal{F}_{\zeta}^{-1}[f]$). Then we have the following results.
 \\
 (i) If $r \nmid (q^n-1)/\Phi_n(q)$, 
 then $\supp(F)$ contains an element of degree $n$ over $\F$;\\
 (ii) If $\supp(F)$ contains an element of degree $n$ over $\F$, then $r \nmid (q^d-1)$ for every positive divisor $d$ of $n$ with $d < n$;\\
 (iii) If $\supp(F)$ contains a primitive element of $\Fn$, then $r = q^n-1$.
 \end{lem}

In particular (i) implies the existence of an irreducible factor of degree $n$ for any polynomial 
$h(x) \in \F[x]$ satisfying a constraint on the least period as follows. 
Here $\Fn^\times$ and $L^\times$ denote the set of all invertible elements in $\Fn$ and $L$ respectively. 

  \begin{lem}\label{lem: factor of deg n}
  Let $q$ be a power of a prime, let $n \geq 2$, let $h(x) \in \F[x]$, and let $L$ be any subfield of $\Fn$ containing the image $h(\Fn^\times)$.
  Define the polynomial
  $$
  S(x) = \left(1 - h(x)^{\#L^\times} \right) \bmod\left( x^{q^n-1} - 1\right) \in \F[x].
  $$
  Write $S(x) = \sum_{i=0}^{q^n-2} s_i x^i$ for some coefficients $s_i \in \F$. 
  If the cyclic sequence $(s_i)_{i=0}^{q^n-2}$ has least period $r$ satisfying $r \nmid (q^n-1)/\Phi_n(q)$, then 
  $h(x)$ has an irreducible factor of degree $n$ over $\F$.
  \end{lem}

\section{Characteristic elementary symmetric and delta functions}\label{section: delta functions}

In this section we apply Lemma \ref{thm: connection} for the purposes of studying digit sums of irreducible polynomials, and thus give in Lemma \ref{lem: gamma and delta functions}
sufficient conditions for an irreducible polynomial to have a prescribed  sum of coefficients.
For this, we first place the characteristic elementary symmetric functions in the context of their DFT, 
which we shall refer to here simply as delta functions. These delta functions are indicators, with values in a finite field, 
for sets of values in $\mathbb{Z}_{q^n-1}$ whose 
canonical integer representatives have certain Hamming weights in their $q$-ary representation and $q$-digits all belonging to the set $\{0,1\}$.
Essentially, characteristic elementary symmetric functions are exponential characteristic generating functions of the sets that the delta functions indicate.

For $\xi \in \Fn$, the characteristic polynomial $h_\xi(x) \in \F[x]$ of degree $n$ over $\F$ with root $\xi$ is given by 
$$
h_\xi(x) = \prod_{k=0}^{n-1}\left( x - \xi^{q^k} \right) = \sum_{w=0}^{n}(-1)^w \sigma_w(\xi) x^{n-w},
$$
where for $0 \leq w \leq n$, $\sigma_w : \Fn \to \F$ is the {\em characteristic elementary symmetric} function given by $\sigma_0(\xi) = 1$ and
$$
 \sigma_w(\xi) = \sum_{0 \leq i_1 <  \cdots < i_w \leq n-1} \xi^{q^{i_1} + \cdots + q^{i_w}}, 
 $$
 for $ 1 \leq w \leq n$.
 In particular $\sigma_1 = \Tn$ is the (linear) trace function and $\sigma_n = N_{\Fn/\F}$ is the (multiplicative) norm function. 
 Whenever $q = 2$ and $\xi \neq 0$, then $\sigma_0(\xi) = \sigma_n(\xi) = 1$ always. 
 If $\xi \neq 0$, then (in general) $h_{\xi^{-1}}(x) = (-1)^n \sigma_n(\xi^{-1}) x^n h_\xi(1/x) = h_\xi^*(x)$, where $h_\xi^*(x)$ is the (monic) {\em reciprocal} of $h_\xi(x)$. Thus
 $\sigma_w(\xi) = \sigma_n(\xi) \sigma_{n-w}(\xi^{-1})$. Clearly $h_\xi(x)$ is irreducible if and only if so is $h_{\xi}^*(x)$. 
 This occurs if and only if $\deg_{\F}(\xi) = n$.
 
Next we introduce the characteristic delta functions and the sets they indicate. But
first let us clarify some ambiguity in our notation: For $a,b \in \mathbb{Z}$, we denote by $a \bmod b$ the remainder of division of $a$ by $b$.
That is, $a \bmod b$ is the smallest integer $c$ in $\{0, 1, \ldots, b-1\}$ that is congruent to $a$ modulo $b$, and write $c = a\bmod b$. 
Similarly if $\bar{a} = a + b\mathbb{Z}$ is an element of $\mathbb{Z}_b$, we use the notation $\bar{a} \bmod b := a \bmod b$ to express the canonical representative of $\bar{a}$ in $\mathbb{Z}$.
But we keep the usual notation $k \equiv a \pmod{b}$ to state that $b \mid (k-a)$.
 
 For $w \in [0, n] := \{0, 1, \ldots, n\}$, define the sets $\Omega(w) \subseteq \mathbb{Z}_{q^n-1}$ by $\Omega(0) = \{0\}$ and 
 $$
 \Omega(w) = \left\{ k \in \mathbb{Z}_{q^n-1} \ : \ k \bmod (q^n-1) = q^{i_1} + \cdots + q^{i_w}, \ 0 \leq i_1 <  \cdots < i_w \leq n-1              \right\}
$$
for $1 \leq w \leq n$. That is $\Omega(w)$ consists of all the elements $k \in \mathbb{Z}_{q^n-1}$ whose canonical representatives in $\{0, 1, \ldots, q^n-2\} \subset \mathbb{Z}$ 
have Hamming weight $w$ in their 
$q$-ary representation $(a_{n-1}, \ldots, a_0)_q$, with each $a_i \in \{0,1\}$.
Note this last condition that each $a_i \in \{0,1\}$ is automatically redundant when $q = 2$, since in general each $a_i \in [0, q-1]$ in the $q$-ary representation $t = (a_{m}, \ldots, a_0)_q$ of
a non-negative integer
$
t = \sum_{i=0}^{m}a_i q^i.
$

When $q = 2$, note $\Omega(n) = \emptyset$ since there is no integer in $\{0, 1, \ldots, 2^n-2\}$ with Hamming weight $n$ in its binary representation. 
Observe also that $|\Omega(w)| = {n \choose w}$ for each $0 \leq w \leq n$, unless $(q,w) = (2,n)$. 
Moreover $\Omega(v) \cap \Omega(w) = \emptyset$ whenever $v \neq w$, by the uniqueness of base representation of integers. 
We extend the domain of $\Omega$ to sets by setting, for $W \subseteq [0, n]$,
$$
\Omega(W) = \bigsqcup_{w \in W} \Omega(w).
$$
Define also the {\em reflection} of $W$ to be the set $n - W := \{n - w \ : \ w \in W\} \subseteq [0, n]$. 
Clearly $|\Omega(W)| = |\Omega(n - W)| = \sum_{w \in W}{n \choose w}$ if $(q, W) \neq (2,\{0\}), (2, \{n\})$. 

For $W \subseteq [0, n]$, define the characteristic (finite field valued) function
$\delta_W : \mathbb{Z}_{q^n-1} \to \mathbb{F}_p$ of the set $\Omega(W)$ by
$$
\delta_W(k) = \begin{cases}
              1 & \mbox{ if } k \in \Omega(W);\\
              0 & \mbox{ otherwise.}
              \end{cases}
              $$
Note $\delta_W = \sum_{w \in W} \delta_w$.
If $W = \{w\}$ contains only a single element, we simply write $\delta_w$ instead of $\delta_{\{w\}}$. 
Observe that our $\delta_0$ is the Kronecker delta function on $\mathbb{Z}_{q^n-1}$ 
with values in $\{0,1\} \subseteq \mathbb{F}_p$.

Let $\zeta$ be a primitive element of $\Fn$ and let $w \in [0, n]$. Then   $\sigma_0(\zeta^k) = 1$ for each $k$ and so $\sigma_0(\zeta^k) = \mathcal{F}_\zeta[\delta_0](k)$. 
Now let $ 1 \leq w \leq n$.  With the extra assumption that $(q,w) \neq (2,n)$, we have
\begin{align*}
 \sigma_w(\zeta^k) &= \sum_{0 \leq i_1 <  \cdots < i_w \leq n-1} \zeta^{k\left(q^{i_1} + \cdots  + q^{i_w}\right)} = \sum_{j \in \mathbb{Z}_{q^n-1}} \delta_w(j) \zeta^{kj}  = \mathcal{F}_\zeta[\delta_w](k).
\end{align*}
 
This derives the following useful result. 

\begin{lem}\label{lem: sigma delta}
 Let $\zeta$ be a primitive element of $\Fn$ and let $w \in [0, n]$. If $q = 2$, further assume that $w \neq n$. Then   
 $$
 \sigma_w(\zeta^k) = \mathcal{F}_\zeta[\delta_w](k), \hspace{2em} k \in \mathbb{Z}_{q^n-1}.
 $$
\end{lem}

The following lemma gives sufficient conditions for the existence of irreducible polynomials of degree $n$ satisfying the desired constraints on their coefficients.

\begin{lem}\label{lem: gamma and delta functions}
 Let $q$ be a power of a prime, let $c \in \F$, let $n \geq 2$ and let $W \subseteq [0,n]$. 
 If $q = 2$, further assume that $n \not\in W$. We have the following two results. 
 \\\\
 (i) If the least period of the function $\gamma_{W,c} : \mathbb{Z}_{q^n-1} \to \F$, given by 
 $$
 \gamma_{W,c} = \sum_{w \in W}(-1)^w \delta_w - c \delta_0,
 $$
 is not a divisor of $(q^n-1)/\Phi_n(q)$, then there exists an irreducible polynomial $P(x)$ of degree $n$ over $\F$ such that $S_{n-W}(P) \neq c$.
  \\\\
 (ii) If the least period of the function $\Delta_{W, c} : \mathbb{Z}_{q^n-1} \to \F$, given by
 $$
 \Delta_{W,c} =  \delta_0 - \left( \sum_{w \in W} (-1)^w \delta_w - c \delta_0\right)^{\otimes(q-1)},
 $$
 is not a divisor of $(q^n-1)/\Phi_n(q)$, then there exists an irreducible polynomial $P(x)$ of degree $n$ over $\F$ such that $S_{n-W}(P) = c$.
 \end{lem}

\begin{proof}

First fix a primitive element $\zeta$ of $\Fn$. 

(i) Define the function $\widehat{\gamma}_{W, c} : \mathbb{Z}_{q^n-1} \to \F$ by 
$$
\widehat{\gamma}_{W, c}(k) = \sum_{w \in W}(-1)^w\sigma_w(\zeta^k) - c.
$$
By Lemma \ref{lem: sigma delta}, the linearity of the DFT, and the fact that $c = \mathcal{F}_\zeta[c \delta_0]$, we have
\begin{align*}
\widehat{\gamma}_{W, c} &= \sum_{w \in W}(-1)^w\mathcal{F}_\zeta[\delta_w] - \mathcal{F}_\zeta[c \delta_0]
 = \mathcal{F}_\zeta \left[ \sum_{w \in W} (-1)^w\delta_w - c \delta_0\right]
 = \mathcal{F}_\zeta \left[ \sum_{w \in W} (-1)^w\delta_w - c \delta_0\right]\\
 &=  \mathcal{F}_\zeta \left[ \gamma_{W,c}\right].
\end{align*}
Thus $\gamma_{W,c} = \mathcal{F}_\zeta^{-1}[\widehat{\gamma}_{W, c}]$. 
Let $F : \mathbb{F}_{q^n} \to \F$ be the associate function of $\widehat{\gamma}_{W, c}$ defined by $F(\zeta^k) = \widehat{\gamma}_{W, c}(k)$ (and say $F(0) = 0$).
Since $\gamma_{W,c} = \mathcal{F}_\zeta^{-1}[\widehat{\gamma}_{W, c}]$ has the desired least period by assumption, Lemma \ref{thm: connection} implies there exists an element $\xi \in \Fn$
of degree $n$ over $\F$ such that $F(\xi) \neq 0$, i.e., $\sum_{w \in W}\sigma(\xi) \neq c$.

(ii) Consider the function $\widehat{\Delta}_{W, c} : \mathbb{Z}_{q^n-1} \to \mathbb{F}_q$ defined by
$$
\widehat{\Delta}_{W, c}(k) = 1 - \left( \sum_{w \in W}\left[ x^{n-w} \right] h_{\zeta^k}(x) - c \right)^{q-1},
$$
where $h_{\zeta^k}(x)$ is the characteristic polynomial of degree $n$ over $\F$ with root $\zeta^k$. Note that
$$
\widehat{\Delta}_{W, c}(k) = \begin{cases}
                             1 &  \mbox{ if } \sum_{w \in W}\left[ x^{n-w} \right] h_{\zeta^k}(x) = c;\\
                             0 & \mbox{ otherwise}.
                            \end{cases}
$$
By Lemma \ref{lem: sigma delta} and the linearity of the DFT, we have
\begin{align*}
 \sum_{w \in W}\left[x^{n-w} \right] h_{\zeta^k}(x) &= \sum_{w \in W}(-1)^w \sigma_w(\zeta^k) = \sum_{w \in W}(-1)^w \mathcal{F}_{\zeta}[\delta_w](k)\\
 &= \mathcal{F}_\zeta\left[ \sum_{w \in W} (-1)^w \delta_w \right](k).
\end{align*}
Clearly $c = \mathcal{F}_\zeta[c \delta_0]$ and particularly $1 = \mathcal{F}_\zeta[\delta_0]$. It follows that
\begin{align*}
\widehat{\Delta}_{W, c} &= \mathcal{F}_\zeta[\delta_0] - \left( \mathcal{F}_\zeta\left[ \sum_{w \in W} (-1)^w \delta_w -c \delta_0 \right]  \right)^{q-1}\\
&= \mathcal{F}_\zeta[\delta_0] -\mathcal{F}_\zeta\left[ \left(\sum_{w \in W} (-1)^w \delta_w -c \delta_0 \right)^{\otimes(q-1)}   \right]\\
&= \mathcal{F}_\zeta\left[  \delta_0 - \left(\sum_{w \in W} (-1)^w \delta_w -c \delta_0 \right)^{\otimes(q-1)}  \right]\\
&= \mathcal{F}_\zeta\left[ \Delta_{W,c} \right].
\end{align*}
Hence $\Delta_{W,c} = \mathcal{F}_\zeta^{-1}[\widehat{\Delta}_{W, c}]$. Let $F : \Fn \to \mathbb{F}_q$ be the associate function of $\widehat{\Delta}_{W,c}$ defined by 
$F(\zeta^k) = \widehat{\Delta}_{W,c}(k)$ (and say $F(0) = 0$). Since $\Delta_{W,c} = \mathcal{F}_\zeta^{-1}[\widehat{\Delta}_{W, c}]$ has the desired least period by assumption, Lemma \ref{thm: connection}
implies there exists an element $\xi$ of degree $n$ over $\F$ such that 
$$
0 \neq F(\xi) = 1 - \left( \sum_{w \in W}\left[ x^{n-w} \right] h_{\xi}(x) - c \right)^{q-1}.
$$
Thus $F(\xi) = 1$, $h_\xi(x)$ is irreducible of degree $n$ over $\F$, and $\sum_{w \in W}\left[ x^{n-w} \right] h_{\xi}(x) = c$. This concludes the proof of (ii).
 \end{proof}

 \section{Least period of delta functions and proof of main results}\label{section: period of delta functions}
 
 Having obtained the sufficient condition in Lemma \ref{lem: gamma and delta functions}, we proceed to prove the result in Lemma \ref{lem: period 1}
 that the sums of delta functions have, except for a few clear exceptions, maximum least period. 
 The proof of this is of a rather elementary although constructive type nature.
 We then conclude the section with proofs of the main results. 
 
First we introduce some notations. 
For a set $S \subseteq \mathbb{N}_0$, let $S^+  = S\setminus\{0\}$.
For a non-negative integer $t = \sum_{i \in A} a_i q^i$ with $A \subset \mathbb{N}_0$ finite and each $a_i \in [1, q-1]$, 
we call $A$ the {\em $q$-support} of $t$ and write $\supp_q(t) = A$. 
We let $w_q(t) = |A|$ be the Hamming weight of $t$ in its $q$-ary representation. 
Recall that $n - S := \{n - s \mid s \in S\}$. Recall also that we at times identify elements in $\Zqn$ with integers in the natural way and vice versa (with addition taken modulo $q^n-1$).
This endows $\Zqn$ with the natural ordering in $\mathbb{Z}$.  
For the sake of brevity we use the notation $Q_n = (q^n-1)/(q-1)$. The following lemma will be useful in the proof of Lemma \ref{lem: period 1}. 

\begin{lem}\label{lem: period 0}
 Let $S \subseteq \{0, 1, \ldots, n\}, S \neq \emptyset$. If $q = 2$, further assume $0, n \not\in S$. 
 Then for all $i \in \Zqn$, we have $\delta_{n - S}(i) = \delta_S(Q_n - i)$. 
 Hence $\delta_{n - S}$ is a shift of the reversal of $\delta_S$. 
\end{lem}

\begin{proof}
If we identify $i$ with its canonical representative (which we can) we have $\delta_S(Q_n - i ) = 1$ if and only if 
$Q_n - i = \sum_{k \in A} q^k$ for some $A \subseteq [0, n-1]$ with $|A| \in S$, i.e., 
 $i = \sum_{k \in [0, n-1]\setminus A} q^k$. 
 This occurs if and only if $\delta_{n - S}(i) = 1$. 
\end{proof}

\begin{lem}\label{lem: period 1}
 Let $n \in \mathbb{N}$ and $S \subseteq \{0,1,\ldots, n\}$, $S \neq \emptyset$. 
 If $q = 2$, assume that $n \not\in S$ and $S \neq \{0, 1, \ldots, n-1\}$. 
 If $q = 3$, further assume $S \neq \{0, n\}$. 
 Then $\delta_S$ has maximum least period $q^n-1$.
\end{lem}

\begin{proof}
The cases when $n = 1$ are easy to check so we assume that $n > 1$. 
Note $\delta_0$ has Hamming weight $1$ and thus least period $q^n-1$. 
Similarly when $q > 2$, $\delta_n$ has Hamming weight $1$ (since $\delta_n(k) \neq 0$ if and only if $k \bmod (q^n-1) = Q_n$ for $q > 2$); 
hence it has least period $q^n-1$. 
If $q > 2$, 
then
$\delta_{\{0,n\}}$ has exactly two runs of `$0$'s (since $n > 1$). Their lengths are $Q_n - 1$ and $q^n-2 - Q_n$, respectively.  
It is easy to check these lengths are distinct for $q > 3$; it follows $\delta_{\{0,n\}}$ has least period $q^n-1$, when $q > 3$. 
We may now assume $S^+ \neq \emptyset$ and $\min(S^+) < n$.

Let $r$ be the least period of $\delta_S$. Necessarily $1 \leq r \mid (q^n-1)$. 
Write $r = \sum_{i \in E} r_i q^i$ for some (non-empty) subset $E \subseteq [0, n-1]$ and some integers $r_i$ with $1 \leq r_i \leq q-1$, for $i \in E$.

There are some technical differences between the cases when $q = 2$ and $q > 2$; we will treat these two separately.
\\

{\bf Case 1.} Assume $q = 2$. 
Let us suppose, by way of contradiction, that $r < 2^n-1$. Hence $|E| \leq n-1$.
Since $r \mid 2^n-1$ is odd, then $m := \min([0, n-1] \setminus E) \geq 1$. By definition, $m \notin E$ but $m -1\in E$. 
Define the integer $\zeta := r + 2^{m -1} = 2^m + \sum_{i \in E \setminus\{m-1\}} 2^i$, of weight $|E|$. 

We can assume there exists $s \in S$ with $s \geq |E|$, otherwise we consider instead the complement of $\delta_S$ (obtained by interchanging `1's and `0's);
it has equal least period and corresponding set $S$ satisfying $s \geq |E|$ for some $s \in S$. 
Now let $s \in [|E|, n-1] \cap S$. Note that $n - (|E| + 1) \geq s - |E|$. 
Hence we can find $\beta \in [0, 2^n-1)$ with $w_2(\beta) = s - |E|$ and $\supp_2(\beta) \cap(E \cup \{m\}) = \emptyset$. Let $\alpha = \beta + \zeta$ for any such $\beta$. 
By construction $w_2(\alpha) = s$. Then $\alpha - r = \beta + (\zeta - r) = \beta + 2^{m-1}$ has weight $s - |E| + 1 \in S$.

Since $n - (s - |E| + 1) = n - s -1 + |E| \geq |E|$, we can find $k \in (0, 2^n-1)$ of weight $s - |E| + 1 \in S$ with $\supp_2(k) \cap E = \emptyset$. 
Note that $w_2((k + r) \bmod{(2^n-1)}) = (s + 1) \bmod{n}$. It is clear that $\delta_S(k + r) = \delta_S(k) = 1$. Then $(s + 1) \bmod{n} \in S$. It follows that $s, s+1, s+2, \ldots, n-1, 0 \in S$. 

Consider the complement $\delta_T$ of $\delta_S$ with equal least period and corresponding set of weights $T = [0, n-1] \setminus S$.
Since $S \neq [0, n-1]$ by assumption, $T \neq \emptyset$.
Now if there exists $t \in T$ such that $t \geq |E|$, then, 
similarly as we did before, we obtain $t, t+1, \ldots, n-1, 0 \in T$. In particular $0 \in S \cap T$, a contradiction. Hence $t < |E|$ for all $t \in T$. 
Particularly $\max(T) < |E|$. 

Let $M \subset E$ with $|M| = \max(T)$ and define $\gamma := \sum_{i \in M} 2^i$. 
Clearly $\delta_T(\gamma) = 1$. Note $\gamma < r$ and $\gamma + 2^n-1 - r \in (0, 2^n-1)$. Since $r \mid (2^n-1 - r)$, then 
$\delta_T(\gamma + 2^n-1 - r) = \delta_T(\gamma) = 1$. Hence $w_2(\gamma + 2^n-1 - r) \in T$. However
$$
\gamma + 2^n-1 - r = \sum_{i \in M} 2^i + \sum_{i \in [0, n-1] \setminus E} 2^i 
$$
has weight $|M| + n - |E| = \max(T) + n - |E| > \max(T)$, a contradiction. This concludes the proof for the case when $q = 2$.  
\\

{\bf Case 2.} Let $q > 2$. 
We may assume $\min(S^+) \leq n/2$. Indeed, otherwise we consider instead $\delta_{n - S}$, with $\min((n - S)^+) < n/2$. 
Since $\delta_{n - S}$ is a shift of the reversal of $\delta_S$ by Lemma \ref{lem: period 0}, it has the same least period as $\delta_S$.  

We claim that $|E| > \min(S^+)$. Indeed, suppose on the contrary that $|E| \leq \min(S^+)$. 
To obtain a contradiction, first we show that 
$r_i = 1$ for all $i \in E$. For this, note that if $0 \in S$, then $ \delta_S(r) = \delta_S(0 + r) = \delta_S(0) = 1$ implies that $r_i = 1$ for all $i \in E$.
Now assume $0 \not\in S$. 
Since $n - |E| \geq n - \min(S^+) \geq \min(S^+)$, 
there exists a subset $L \subseteq [0,n-1] \setminus E$
 with $|L| = \min(S^+)$. Pick one such subset $L$ and define $\alpha := \sum_{i \in L} q^i$. Clearly $\delta_S(\alpha) = 1$. 
 Since $\alpha + r = \sum_{i \in L}q^i + \sum_{i \in E} r_i q^i$ with $L \cap E = \emptyset$ and $|L| + |E| = \min(S^+) + |E| \leq 2\min(S^+) \leq n$, 
 it follows that $\alpha + r \in (0, q^n - 1]$. Because $\delta_S(\alpha + r) = 1$ but $0 \notin S$, 
 then $\alpha + r < q^n-1$ strictly. Consequently $r_i = 1$ for all $i \in E$.
 
 Let $M \subset [0, n-1]$ such that $|M| = \min(S^+)$ and $E \subseteq M$. 
 Define $\beta := \sum_{i \in M} q^i$. Then $\delta_S(\beta - r) = \delta_S(\beta) = 1$, where
 $\beta - r = \sum_{i \in M \setminus E} q^i$. 
 Thus $|M| - |E| = \min(S^+) - |E| \in S$. It is impossible that $\min(S^+) - |E| \in S^+$, otherwise the minimality of $\min(S^+)$ is contradicted.
 Necessarily $|E| = \min(S^+)$
and $|E| \in S$. Since $r_i = 1$ for all $i \in E$ as well, we get $\delta_S(r) = 1$. Then $\delta_S(2r) = 1$. 
 However the assumptions $q > 2$ and $|E| = \min(S^+) \leq n/2$ imply $\delta_S(2r) = 0$ since $2r = \sum_{i \in E} 2 q^i < q^n-1$ has $q$-digits not in $\{0,1\}$. We thus obtain a contradiction.
 The claim follows. 

Let $\ell \in S^+$
 such that $|E| > \ell$, and let $G \subset E$ with $|G| = \ell$. 
 Define $\gamma := \sum_{i \in G} q^i$. Clearly $\delta_S(\gamma) = 1$. Since $r \mid (q^n-1 - r)$, we also have $\delta_S(\gamma + q^n - 1 - r) = 1$. 
 Note that $\gamma < r$ and
 $\gamma + q^n-1 - r \in (0, q^n-1)$. Moreover
 \begin{align*}
 q^n-1 + \gamma - r &= \sum_{i \in [0, n-1]} (q-1)q^i + \sum_{i \in G}(1 - r_i)q^i - \sum_{i \in E \setminus G} r_i q^i\\
 &= \sum_{i \in G}(q - r_i)q^i + \sum_{i \in E \setminus G}(q-1 - r_i)q^i + \sum_{i \in [0,n-1] \setminus E} (q-1)q^i.
 \end{align*}
 Because the three subsets $G, E \setminus G, [0, n-1] \setminus E \subset [0, n-1]$ are pairwise disjoint, and each of the coefficients of the $q$-powers belongs to the set $[0, q-1]$,
 the above is the $q$-ary representation of $\gamma + q^n-1 - r \in (0, q^n-1)$. 
 Since $\delta_S(\gamma + q^n-1 - r) = 1$, the coefficients of the $q$-powers above must all belong to the set $\{0,1\}$. Because $q > 2$, i.e., $q-1 > 1$,
 it follows from the equation above that $E = [0, n-1]$ and $r_i = q-1$ for all $i \in G$. Since $G \subset E = [0, n-1]$ is arbitrary (but of size $\ell > 0$) 
 we get that $r_i = q-1$ for all $i \in [0, n-1]$.
 Thus $r = \sum_{i \in [0, n-1]}(q-1) q^i = q^n-1$. 
\end{proof}

Finally we are ready to prove Theorems \ref{thm: main q = 2} and \ref{thm: main q > 2}.

\begin{proof}[{\bf Proof of Theorem \ref{thm: main q = 2}}]
The exceptions are explained by the fact that 
the number of non-zero coefficients in an irreducible polynomial over $\mathbb{F}_2$ must be odd, and its leading, constant terms (for $n \geq 2$) have coefficient $1$.

Assume $(c,W)$ is none of the exceptions listed in (\ref{eqn: exc1}). 
First we consider the case when $0 \not\in W$, i.e., $W \subseteq [1,n]$.
Now let $R = n-W$ be the reflection of $W$. 
Clearly $R \neq \emptyset$ and 
\begin{equation}\label{eqn: exc}
R \subseteq [0, n-1]; \text{ particularly }  (c, R) \neq (0, \{0\}), \ (0, [1, n-1]), \ (1, [0, n-1]).
\end{equation}
For $q = 2$,  
the $\Delta_{R,c}$ function in Lemma \ref{lem: gamma and delta functions} (ii) becomes
$\Delta_{R,c} = \delta_R + (c+1)\delta_0 = \delta_S$, where
 $$
 S := \begin{cases}
                                       R & \mbox{ if } c = 1;\\
                                       R \cup \{0\} & \mbox{ if } c = 0 \text{ and } 0 \not\in R;\\
                                       R \setminus \{0\} & \mbox{ if } c = 0 \text{ and } 0 \in R.
                                      \end{cases}
$$
In each of these cases, $S \subsetneq [0, n-1]$ with $S \neq \emptyset$ (which follows from (\ref{eqn: exc})). By Lemma \ref{lem: period 1}, $\Delta_{R,c} = \delta_S$ has 
maximum least period $2^n-1$. Then Lemma \ref{lem: gamma and delta functions} implies there exists an irreducible polynomial $P(x)$ of degree $n$ over $\mathbb{F}_2$ such that
$c = \sum_{v \in R}[x^{n - v}]P(x) = \sum_{w \in W}[x^w]P(x)$. Thus the result holds when $W \subseteq [1,n]$ is not any of the exceptions. 
Next we consider the cases when $0 \in W$
and $W$ is not any of the exceptions. The case when $W = \{0\}$ is clear and so we may assume that $W \setminus \{0\} \neq \emptyset$. 
Note for $P(x) \in \mathbb{F}_{2}[x]$ irreducible of degree $n \geq 2$, that $\sum_{w \in W}[x^w]P(x) = c$ if and only if $\sum_{w \in W'}[x^w]P(x) = k$, 
where $W' = W\setminus \{0\}$ and $k = c + 1$.
Thus $W' \subseteq [1,n]$ and since $(c, W)$ is none of the exceptions in (\ref{eqn: exc1}), one can check that $(k, W')$ is also none of the exceptions in (\ref{eqn: exc1}). Result follows from the previous arguments.
\end{proof}

Before we prove Theorem \ref{thm: main q > 2}, we need the following simple fact.

\begin{lem}\label{lem: period by composition}
 Let $N \in \mathbb{N}$, let $A, B$ be non-empty sets, let $f : \mathbb{Z}_N \to A$ and let $\pi: A \to B$. Then the least period of $f$ is at least as large as the least period of $\pi \circ f$.  
\end{lem}

\begin{proof}
Let $r$ be the least period of $f$. Note for every $m \in \mathbb{Z}_N$ we have $\pi \circ f(m + \bar{r}) = \pi \circ f(m)$. Then $\pi \circ f$ is $r$-periodic; hence the least period
of $\pi \circ f$ is at most $r$.
\end{proof}

\begin{proof}[{\bf Proof of Theorem \ref{thm: main q > 2}}]
Let $R = n - W$.
Define the function $\pi : \F \to \{0,1\} \subseteq \mathbb{F}_p$ by $\pi(k) = 1$ if $k \neq 0$ and $\pi(0) = 0$. 
Consider the function $g : \mathbb{Z}_{q^n-1} \to \mathbb{F}_p$ given by $g = \pi \circ \gamma_{R,c}$, where $\gamma_{R,c}$ is the function in Lemma \ref{lem: gamma and delta functions} (i). 
One can check that $g = \delta_S$, where $S \subseteq [0,n]$ is defined as follows.
$$
S = \begin{cases}
    R, & \mbox{ if } c = 0 \text{, or } c \in \F\setminus\{0,1\} \text{ and } 0 \in R;\\
    R \cup \{0\}, & \mbox{ if } c \neq 0 \text{ and } 0 \not\in R;\\
    R\setminus \{0\}, & \mbox{ if } c=1 \text{ and } 0 \in R.
    \end{cases}
$$
By Lemma \ref{lem: period 1}, $\delta_S$ has maximum least period except when $(q, S) = (3, \{0,n\})$. 
If $(q, S) \neq (3, \{0, n\})$, then, by Lemma \ref{lem: gamma and delta functions} (i) together with Lemma \ref{lem: period by composition}, 
there exists an irreducible polynomial $P(x)$ of degree $n$ over $\F$ with $S_W(P) = S_{n - R}(P) \neq c$.
Now consider the case when $(q,S) = (3, \{0,n\})$. By the definition of $S$ and of $R$, we have
either

(i) $W = \{0,n\}$ with $c = 0$ or $c = 2$, or

(ii) $W = \{0\}$ with $c \neq 0$.

The result follows here from the elementary fact that every element of $\F^*$ is the norm of an element of degree $n$ over $\F$, together with the assumption $q > 2$.
\end{proof}

\begin{thebibliography}{115}




\bibitem{Cohen 2004}
S.D. Cohen, {\em Primitive polynomials over small fields},  Finite fields and applications, 197--214, Lecture Notes in Comput. Sci., 2948, Springer, Berlin, 2004. 


\bibitem{cohen 2006}
S.D. Cohen, {\em Primitive polynomials with a prescribed coefficient}, Finite Fields Appl. 12 (2006), no. 3, 425--491.


\bibitem{cohen-presern 2006}
S.D. Cohen and M. Pre\v{s}ern, {\em Primitive polynomials with prescribed second
coefficient}, Glasgow Math. J. 48 (2006), 281--307.


\bibitem{cohen-presern 2008}
S.D. Cohen and M. Pre\v{s}ern, {\em The Hansen-Mullen primitivity conjecture:
completion of proof}, Number theory and polynomials, 89--120, London Math. Soc. Lecture Note Ser., 352, Cambridge Univ. Press, Cambridge, 2008.



\bibitem{Drmota-Rivat-Stoll}
M. Drmota, J. Rivat, T. Stoll, {\em The sum of digits of primes in $\mathbb{Z}[i]$}, Monatshefte f\"ur Mathematik
v.155 no.3, (2008), 317--347.


\bibitem{fan-han}
S.Q. Fan and W.B. Han,
{\em p-Adic formal series and primitive polynomials over finite fields}
Proc. Amer. Math. Soc. 132 (2004),  15--31.

\bibitem{Fitzgerald-Yucas}
R.W. Fitzgerald and J.L. Yucas, {\em Irreducible polynomials over GF(2) with three prescribed coefficients}, Finite Fields Appl. 9 (2003), 286--299.



\bibitem{garefalakis}
T. Garefalakis, {\em Irreducible polynomials with consecutive zero coefficients}, Finite Fields Appl. 14 (2008), no. 1, 201--208.

\bibitem{Ha}
J. Ha, {\em Irreducible polynomials with several prescribed coefficients}, arXiv:1601.06867 [math.NT], preprint (2016).



\bibitem{ham-mullen}
K.H. Ham and G.L. Mullen, {\em Distribution of irreducible polynomials of small degrees over finite fields}, Math. Comp. 67 (1998), no. 221, 337--341.

\bibitem{han}
W.B. Han, {\em On Cohen's problem}, Chinacrypt ’96, Academic Press (China) (1996) 231--235 (Chinese).

\bibitem{hansen-mullen}
T. Hansen and G.L. Mullen, {\em Primitive polynomials over finite fields}, Math. Comp. 59 (1992), 639--643.

\bibitem{Koc}
\c{C}.K. Ko\c{c}, {\em Open Problems in Mathematics and Computational Science}, Springer International Publishing, 2014.


\bibitem{George paper}
D. Panario and G. Tzanakis,
{\em A generalization of the Hansen--Mullen conjecture on irreducible polynomials over finite fields}, Finite Fields Appl. 18 (2) (2012) 303--315.


\bibitem{Mauduit-Rivat}
C. Mauduit, J. Rivat, {\em Sur un probl\`eme de Gelfond : la somme des chiffres des nombres premiers}, Annals of Mathematics, Vol. 171 (2010), No. 3, 1591--1646.

\bibitem{Morgenbesser}
J.F. Morgenbesser, {\em The sum of digits of Gaussian primes}, The Ramanujan Journal, v.27 no.1 (2012) 43--70.

\bibitem{pollack}
P. Pollack,
{\em Irreducible polynomials with several prescribed coefficients}, Finite Fields Appl. 22 (2013) 70--78.

\bibitem{ren}
D-B. Ren, {\em On the coefficients of primitive polynomials over finite fields}, Sichuan
Daxue Xuebao 38, 33--36.

\bibitem{Shparlinski}
I. E. Shparlinski, {\em On primitive polynomials}, Prob. Peredachi Inform. 23 (1988),
100--103 (Russian).

\bibitem{tuxanidy-wang}
A. Tuxanidy, Q. Wang, {\em Composed products and factors of cyclotomic polynomials over finite fields}, Des. Codes Cryptogr.
69 (2013), 203--231.

\bibitem{Tuxanidy-Wang-H-M conjecture}
A. Tuxanidy, Q. Wang. {\em A new proof of the Hansen-Mullen irreducibility conjecture}, preprint (2016).

\bibitem{George thesis}
G. Tzanakis, {\em On the existence of irreducible polynomials with prescribed coefficients over finite fields}, Master’s thesis,
Carleton University, 2010, http://www.math.carleton.ca/~gtzanaki/mscthesis.pdf.

\bibitem{wan}
D. Wan, {\em Generators and irreducible polynomials over finite fields}, Math. Comp. 66 (219) (1997) 1195--1212.





\end{thebibliography}
\end{document}